\documentclass[12pt]{amsart}
\usepackage{a4wide}
\usepackage{amssymb}
\usepackage{amsthm}
\usepackage{amsmath}
\usepackage{amscd}
\usepackage{verbatim}
\usepackage[all]{xy}
\usepackage{longtable, lscape}
\usepackage{todonotes}

\numberwithin{equation}{section}

\theoremstyle{plain}
\newtheorem{theorem}{Theorem}[section]

\newtheorem{lemma}[theorem]{Lemma}
\newtheorem{proposition}[theorem]{Proposition}

\theoremstyle{definition}

\theoremstyle{remark}
\newtheorem{remark}[theorem]{Remark}

\newcommand{\R}{\mathbb{R}}
\newcommand{\Q}{\mathbb{Q}}
\newcommand{\Z}{\mathbb{Z}}

\newcommand{\C}{\mathbb{C}}

\renewcommand{\H}{\mathbb{H}}

\newcommand{\leg}[2]{\left( \frac{#1}{#2} \right)}
\newcommand{\kzxz}[4]{\left(\begin{smallmatrix} #1 & #2 \\ #3 & #4\end{smallmatrix}\right) }
\newcommand{\kabcd}{\kzxz{a}{b}{c}{d}}

\newcommand{\calD}{\mathcal{D}}

\newcommand{\calS}{\mathcal{S}}

\newcommand{\eps}{\varepsilon}
\newcommand{\bs}{\backslash}

\newcommand{\Sl}{\operatorname{SL}}

\newcommand{\SL}{\operatorname{SL}}

\newcommand{\Mp}{\operatorname{Mp}}
\newcommand{\Orth}{\operatorname{O}}

\newcommand{\Hom}{\operatorname{Hom}}

\newcommand{\pr}{\operatorname{pr}}

\newcommand{\SO}{\operatorname{SO}}

\newcommand{\Div}{\operatorname{Div}}

\newcommand{\ord}{\operatorname{ord}}

\begin{document}

\title[Borcherds products with prescribed divisor]{Borcherds products with prescribed divisor}

\author[Jan H.~Bruinier]{Jan
Hendrik Bruinier}
\address{Fachbereich Mathematik,
Technische Universit\"at Darmstadt, Schlossgartenstrasse 7, D--64289
Darmstadt, Germany}
\email{bruinier@mathematik.tu-darmstadt.de}

\thanks{The author is supported by DFG grant BR-2163/4-2.}

\date{\today}

\begin{abstract}
Given an infinite set of special divisors satisfying a mild regularity condition, we prove the existence of a Borcherds product of non-zero weight whose divisor is supported on these special divisors. 
We also show that every meromorphic Borcherds product is the quotient of two holomorphic ones. The proofs of both results rely on the properties of vector valued Eisenstein series for the Weil representation.
\end{abstract}

\maketitle


\section{Introduction and statement of results}
\label{sect:intro}

Let $(L,Q)$ be an even lattice of signature $(n,2)$ with dual $L'$.
%
%
We write $\calD^+$ for the  hermitian symmetric space associated with the connected component of 
the real points of the orthogonal group $\Orth(L)$ of $L$. 
Let $\Gamma\subset \Orth(L)$ be a congruence subgroup which preserves $\calD^+$ and which acts trivially on the discriminant group  $L'/L$. By the theory of  Baily-Borel the quotient 
\[
X_\Gamma = \Gamma\bs \calD^+ 
\]
has a structure as a quasi-projective algebraic variety over $\C$ of dimension $n$. For every $\mu\in L'/L$ and every positive $m\in \Z+Q(\mu)$ there exists a special divisor $Z(m,\mu)$ on $X_\Gamma$. In the projective model of $\calD^+$ it is given by the orthogonal complements of vectors $\lambda\in L+\mu$ with $Q(\lambda)=m$, see \cite{Bo2}, \cite{Ku:Duke}, \cite{Br1}.

We briefly write $L^-$ for the lattice $(L,-Q)$ of signature $(2,n)$.
Recall that there is a Weil representation $\rho_L$ of the metaplectic group $\Mp_2(\Z)$ on the group ring  $\C[L'/L]$, see \cite{Bo1}, \cite{Br1}. By means of the standard $\C$-bilinear pairing on $\C[L'/L]$, the dual representation of $\rho_L$ can be identified with $\rho_{L^-}$.
In his celebrated paper \cite{Bo1},
R.~Borcherds constructed a map from weakly holomorphic modular forms of weight $1-n/2$ for $\Mp_2(\Z)$ with representation $\rho_{L^-}$ to meromorphic modular forms on $X_\Gamma$ whose divisors are supported on special divisors and which have particular infinite product expansions, see \cite[Theorem 13.3]{Bo1} and \cite[Theorem 3.22]{Br1}. Since these Borcherds products give rise to explicit relations among special divisors in the Picard group of $X_\Gamma$, they are of great importance for algebraic and arithmetic applications, see e.g.~\cite{Bo2}, \cite{Ku:Integrals}, \cite{BHY}. 
In this note we prove two useful results about Borcherds products.

We call a set $\calS$ of pairs $(m,\mu)\in \Q_{>0}\times L'/L$ {\em admissible}, if:
\begin{enumerate}
\item For all $(m,\mu)\in \calS$ there exists a $\lambda\in \mu+L$ with $Q(\lambda)=m$.
\item There exists a positive integer $A$ such that $\ord_p(m)\leq A$ for all $(m,\mu)\in \calS$ and for all primes $p$ dividing $2|L'/L|$.
\end{enumerate}
The first condition is equivalent to requiring that $Z(m,\mu)$ be a non-trivial divisor on $X_\Gamma$.

\begin{theorem}
\label{thm:density}
Assume that $n\geq 2$.
Let $\calS$ be an infinite admissible set of pairs 
$(m,\mu)\in \Q_{>0}\times L'/L$ 
Then there exists a Borcherds product $\Psi$ of non-zero weight whose divisor
is supported on divisors $Z(m,\mu)$
 with $(m,\mu) \in \calS$.
\end{theorem}

\begin{remark} 
%
It is much easier to see that there is also a (non-constant) Borcherds product of 
weight $0$ whose divisor
is supported on divisors $Z(m,\mu)$
 with $(m,\mu) \in \calS$.
\end{remark}

This result can be used to construct sections of a {\em non-trivial} power of the tautological bundle over $X_\Gamma$ whose divisor 
is supported on divisors $Z(m,\mu)$
 with $(m,\mu) \in \calS$.
This is employed in the recent proof of the averaged 
Colmez conjecture by Andreatta, Goren, Howard, and Madapusi Pera \cite[Theorem 9.5.5]{AGHM}.

\begin{theorem}
\label{prop:quot}
Assume that $n\geq 1$.
Every Borcherds product for $\Gamma$ is the quotient of two Borcherds products for $\Gamma$ which are holomorphic on $X_\Gamma$.
\end{theorem}

This theorem is useful to reduce statements about Fourier expansions of Borcherds products to the holomorphic case. A slight variant (see Theorem \ref{prop:quot2}), 
together with \cite[Theorem 6.3]{HM}, can be employed to give a different proof of the converse theorem for Borcherds products \cite[Theorem 5.12]{Br1} 
for lattices that split two hyperbolic planes over $\Z$. 
Similar results as Theorems \ref{thm:density} and \ref{prop:quot} were obtained
 in \cite[Section 4]{BBK} for the special case of Hilbert modular surfaces for the full Hilbert modular group.

I thank S.~Ehlen and B.~Howard for useful conversations on the content of this note. Moreover, I thank the anonymous referee for helpful comments.

\section{Preliminaries}
\label{sect:2}

We begin by fixing some general notation.
If $D\in \Z\setminus\{0\}$ is a discriminant, we write $\chi_D$ for the Dirichlet character $\chi_D(a)= \leg{D}{a}$.
If $a$ is a positive integer and $\chi$ is a Dirichlet character,
we denote by $\sigma_s(a,\chi)$ the divisor sum
\[
\sigma_s(a,\chi)= \sum_{d\mid a} \chi(d) d^s.
\]
If $\chi=\chi_1$ is the trivial character modulo $1$, we briefly write $\sigma_s(a)=\sigma_s(a,\chi_1)$. As usual the Moebius function is denoted by $a\mapsto \mu(a)$.

In this section we temporarily consider an even lattice $(L,Q)$ of arbitrary signature $(b^+,b^-)$. We write $N$ for the level of $L$ and $\det(L)$ for the Gram determinant of $L$. Recall that $|\det(L)|=|L'/L|$ and that $N$ and $\det(L)$ have the same prime divisors. Moreover, we denote by $r(L)$ the Witt rank of $L$, i.e., the rank of a maximal totally isotropic sublattice.

As in \cite{Bo1} we denote by $\Mp_2(\Z)$ the metaplectic extension of $\SL_2(\Z)$, realized by the two possible choices of a holomorphic square root $\sigma(\tau)$ of the automorphy factor $j(g,\tau)=c\tau +d$ of $g=\kabcd\in \Sl_2(\Z)$ for $\tau$ in the upper complex half plane $\H$.
If $k\in \frac{1}{2}\Z$,   
we write $M^!_k(\rho_L)$ for the space of $\C[L'/L]$-valued weakly holomorphic modular forms of weight $k$ for the group $\Mp_2(\Z)$ with representation $\rho_L$.
The subspaces of holomorphic modular forms and cusp forms are denoted by 
$M_k(\rho_L)$ and $S_k(\rho_L)$, respectively.

\begin{theorem}
\label{thm:intbasis}
The space $M_{k}^!(\rho_{L})$ has a basis of weakly holomorphic modular forms with integral Fourier coefficients.
\end{theorem}

\begin{proof}
This result is a consequence of \cite[Theorem 5.6]{McG}. For the convenience of the reader we briefly explain how it can be deduced.

If $\ell$ is a non-negative integer, we write $M_{k}^{!,\ell}(\rho_{L})$
 for the subspace of $M_{k}^{!}(\rho_{L})$ consisting of those forms whose pole order at the cusp at $\infty$ is less or equal to $\ell$. Hence we have
\[
M_{k}^{!}(\rho_{L})= \bigcup_{\ell\geq 0} M_{k}^{!,\ell}(\rho_{L}).
\]
For every $\ell$ there is an isomorphism of vector spaces
\[
M_{k}^{!,\ell}(\rho_{L})\longrightarrow M_{k+12\ell}(\rho_{L}),\quad f\mapsto \Delta^{\ell}f,
\]
where $\Delta=q\prod_{j\geq 1}(1-q^j)^{24}$ is the usual discriminant function.
According to  \cite[Theorem 5.6]{McG}, the space $M_{k+12\ell}(\rho_{L})$ has a basis of modular forms with integral coefficients.
The inverse image under the above isomorphism defines a basis $B_\ell$ of $M_{k}^{!,\ell}(\rho_{L})$ with integral coefficients.
We obtain a basis  of $M_{k}^{!}(\rho_{L})$ with integral coefficients by taking a maximal linearly independent subset of $\bigcup_{\ell} B_\ell$.
\end{proof}

\subsection{Eisenstein series}
Here we recall some facts about $\C[L'/L]$-valued Eisenstein series from \cite{BK}.
Let $\kappa\in \frac{1}{2}\Z$ with $2\kappa \equiv b^+-b^- \pmod{4}$.
Assume that $\kappa>2$.

Let $\Gamma'_\infty\subset\Mp_2(\Z)$ be the stabilizer of the cusp $\infty$, that is, the subgroup of pairs $(g,\sigma)\in \Mp_2(\Z)$ for which $g$ is of the form $\pm \kzxz{1}{*}{0}{1}$. Let $(\chi_\mu)_{\mu\in L'/L}$ be the standard basis of $\C[L'/L]$. The element $\chi_0\in \C[L'/L]$ 
transforms under $\Gamma'_\infty$ with a character of the center.
The corresponding Eisenstein series
\begin{align*}
E_{\kappa,L}(\tau)&= \sum_{(g,\sigma)\in \Gamma'
_\infty\bs \Mp_2(\Z)} \sigma(\tau)^{-2\kappa} \cdot \big(\rho_L(g,\sigma)^{-1}\chi_0\big)
\end{align*}
defines a holomorphic function in $\tau\in\H$, satisfying the transformation law
\[
E_{\kappa,L}(\gamma\tau) = \sigma(\tau)^{2\kappa} \rho_L(\gamma) E_{\kappa,L}(\tau)
\]
for all $\gamma=(g,\sigma)\in \Mp_2(\Z)$. In particular, $E_{\kappa,L}(\tau)$ belongs to $M_\kappa(\rho_L)$.
It has a Fourier expansion
\[
E_{\kappa,L}(\tau) = \sum_{\mu\in L'/L}\sum_{m\geq 0} e_{\kappa,L}(m,\mu) \cdot q^m \chi_\mu
\]
with coefficients $e_{\kappa,L}(m,\mu)$, and $q=e^{2\pi i\tau}$.
The constant term of $E_{\kappa,L}$ is given by
\[
e_{\kappa,L}(0,\mu)=\begin{cases} 1,&\text{if $\mu=0$,}\\
0,&\text{if $\mu\neq 0$.}
\end{cases} 
\]
This implies that $E_{\kappa,L}$ does not vanish identically.

If $\kappa=2$, the Eisenstein series $E_{\kappa,L}(\tau)$ can be defined similarly using the usual `Hecke trick'. It has the same properties as in the case $\kappa>2$ with the only difference that in the constant term an additional  non-holomorphic contribution (a multiple of $\Im(\tau)^{-1}$) can occur. The coefficients with positive index are still constant (see e.g.~\cite[Section 3]{BrKue}).

The Fourier expansion of this Eisenstein series was computed in \cite{BK} and \cite{KY}. (Note that in \cite{BK} it was worked implicitly with the lattice $(L,-Q)$. Moreover, the Eisenstein series $2E_{\kappa,L}$ was considered.)
A first result is the following.

\begin{proposition}
\label{prop:nonneg}
For all $m\in \Q_{>0}$ and $\mu \in L'/L$ the coefficients $e_{\kappa,L}(m,\mu)$ are rational numbers.
Moreover, the quantity
\[
(-1)^{(2\kappa-b^++b^-)/4}e_{\kappa,L}(m,\mu)
\]
is non-negative.
\end{proposition}

\begin{proof}
The rationality of the coefficients is Corollary 8 in \cite{BK}. The non-negativity follows from Theorem 7 in \cite{BK} by means of standard bounds for Dirichlet $L$-functions in the region of convergence.
\end{proof}

Now we specialize to the case that
$\kappa=(b^++b^-)/2$, still assuming that $\kappa \geq 2$. The condition that $2\kappa \equiv b^+-b^- \pmod{4}$ is then equivalent to requiring that $b^-$ is even.
For $\mu\in L'/L$ we let $d_\mu=\min\{b\in \Z_{>0}:\; b\mu=0\}$ be the order of $\mu$.
If $m\in \Z+Q(\mu)$ we write
\[
N_{m,\mu}(a)= \left|\{ r\in L/aL:\; Q(r+\mu)\equiv m\pmod{a}\}\right|.
\]
This is a multiplicative function in $a$. Theorem 11 of \cite{BK} gives the following explicit formulas for the Fourier coefficients of $E_{\kappa,L}$.

\begin{theorem}
\label{thm:eiseven}
Assume that $b^++b^-$ is even. Let $\mu\in L'/L$, and let $m\in \Z+Q(\mu)$ be positive. Then
\begin{align*}
e_{\kappa,L}(m,\mu)&= \frac{(2\pi)^{\kappa} m^{\kappa-1}(-1)^{b^-/2}}{\sqrt{|L'/L|}\Gamma(\kappa)}
\cdot \frac{\sigma_{1-\kappa}(d_\mu^2m,\chi_{4D})}{L(\kappa,\chi_{4D})}\prod_{\substack{\text{$p$ prime}\\p\mid 2N}} \frac{N_{m,\mu}(p^{w_p})}{p^{(2\kappa-1)w_p}},
\end{align*}
where $D$ denotes the discriminant $D=(-1)^{(b^++b^-)/2}\det(L)$,
and
\[
w_p=w_p(m,\mu)= 1+2\ord_p(2d_\mu m).
\]
\end{theorem}

\begin{theorem}
\label{thm:eisodd}
Assume that $b^++b^-$ is odd. Let $\mu\in L'/L$, and let $m\in \Z+Q(\mu)$ be positive.
Write $md_\mu^2= m_0 f^2$ for positive integers $m_0,f$ with $(f,2N)=1$ and $\ord_p(m_0)\in \{0,1\}$ for all primes $p$ coprime to $2N$.
Then
\begin{align*}
e_{\kappa,L}(m,\mu)&= \frac{(2\pi)^{\kappa} m^{\kappa-1}(-1)^{b^-/2}}{\sqrt{|L'/L|}\Gamma(\kappa)}
\cdot \frac{L(\kappa-1/2,\chi_{D'})}{\zeta(2\kappa-1)}\\
&\phantom{=}{}\times \sum_{d\mid f} \mu(d)\chi_{D'}(d)d^{1/2-\kappa}\sigma_{2-2\kappa}(f/d)
\prod_{\substack{\text{$p$ prime}\\p\mid 2N}} \frac{N_{m,\mu}(p^{w_p})}{(1-p^{1-2\kappa})p^{(2\kappa-1)w_p}},
\end{align*}
where $D'$ denotes the discriminant $D'=2(-1)^{(b^++b^-+1)/2}m_0\det(L)$,
and
\[
w_p=w_p(m,\mu)= 1+2\ord_p(2d_\mu m).
\]
\end{theorem}

From this result we obtain the following lower bound for the coefficients.

\begin{proposition}
\label{prop:eislb}
Assume that $\kappa=\frac{b^++b^-}{2}>2$, and let
$A\geq 0$. There exists a constant $C>0$ (depending only on $A$ and $L$)
such that for all $(m,\mu)\in \Q_{>0}\times L'/L$  satisfying
\begin{enumerate}
\item[(i)] $m$ is represented by $L+\mu$,
\item[(ii)] $\ord_p(m)\leq A$ for all primes $p$ dividing $2N$,
\end{enumerate}
we have
\[
(-1)^{b^-/2}e_{\kappa,L}(m,\mu)>C\cdot m^{\kappa-1}.
\]
\end{proposition}

\begin{proof}
This is a direct consequence of Theorem \ref{thm:eiseven} and Theorem \ref{thm:eisodd}, combined with elementary estimates for $L$-functions of quadratic Dirichlet characters.

For instance, in the case when $n$ is odd, we have
\[
L(\kappa-1/2,\chi_{D'})>\frac{\zeta(2\kappa-1)}{\zeta(\kappa-1/2)}.
\]
This follows from the Euler product expansion which converges absolutely since $\kappa\geq 5/2$.
In the sum over the divisors of $f$, the term $d=1$ is dominating. Using again the fact that $\kappa\geq 5/2$, we obtain
\begin{align*}
\sum_{d\mid f} \mu(d)\chi_{D'}(d)d^{1/2-\kappa}\sigma_{2-2\kappa}(f/d)
&\geq 1-\sum_{\substack{d\mid f\\ d>1}} d^{1/2-\kappa}\sigma_{2-2\kappa}(f/d)\\
&\geq 1-(\zeta(2)-1)\zeta(3)\\
&\geq 1/5.
\end{align*}
Finally, the condition (i) implies that the representation numbers $N_{m,\mu}(a)$ modulo $a$ are all at least $1$. Condition (ii) implies that for primes $p$ dividing $2N$ the quantities $w_p$ are bounded by $1+2(\ord_p(2N)+A)$. Hence
\[
\prod_{\substack{\text{$p$ prime}\\p\mid 2N}} \frac{N_{m,\mu}(p^{w_p})}{p^{(2\kappa-1)w_p}}
\]
is greater than a positive constant. This concludes the proof of the proposition.
\end{proof}

\begin{remark}
\label{rem:eislb}
If $\kappa=2$ then the assertion of Proposition \ref{prop:eislb} is still true with the slightly weaker lower bound
\[
(-1)^{b^-/2}e_{\kappa,L}(m,\mu)>C\cdot m^{\kappa-1-\eps}
\]
for any $\eps>0$, and a constant $C$ depending in addition on $\eps$. Here the extra $m^{-\eps}$ term comes from bounding $\sigma_{1-\kappa}(d_\mu^2m,\chi_{4D})$ in this case.
\end{remark}

\section{Proofs}

Here we turn to the proofs of the theorems stated in the introduction.

\subsection{Weakly holomorphic modular forms}

Throughout this subsection we assume  that $L$ has signature $(n,2)$
with $n\geq 2$
and put $\kappa=1+n/2$.  

Let $\calS$ be an infinite admissible set of pairs $(m,\mu)\in \Q_{>0}\times L'/L$ as in the introduction.
In view of  of \cite[Theorem 13.3]{Bo1}, Theorem \ref{thm:density} of the introduction is
a consequence of the following proposition.

\begin{proposition}
\label{prop:keyprop}
There exists a weakly holomorphic modular form $f\in M_{2-\kappa}^!(\rho_{L^-})$ with integral Fourier coefficients $c_f(l,\nu)$ with the properties:
\begin{itemize}
\item[(i)]
if $(m,\mu)\in \Q_{>0}\times L'/L$ with $c_f(-m,\mu)\neq 0$, then $(m,\mu)\in \calS$,
\item[(ii)]
$c_f(0,0)\neq 0$.
\end{itemize}
\end{proposition}

To prove this proposition, we recall the following result from
\cite{Bo2} (see also Corollary~3.9 in \cite{BF}).

\begin{proposition}
\label{prop:crit}
There exists a weakly holomorphic modular form $f\in M_{2-\kappa}^!(\rho_{L^-})$ with $\Gamma'_\infty$-invariant prescribed principal part
\[
\sum_{\nu\in L'/L}\sum_{l<0} c(l,\nu)\, q^{l}\chi_\nu\in \C[L'/L][q^{-1/N}]
\]
 at the cusp $\infty$ if and only if
\begin{align}
\label{eq:cuspcond}
\sum_{\nu\in L'/L}\sum_{l>0}  c(-l,\nu)\, b(l,\nu)=0
\end{align}
for every $g=\sum_\nu\sum_l b(l,\nu) q^l\chi_\nu\in S_{\kappa}(\rho_L)$. 

If $n>2$ or $n=2>r(L)$, the constant term
$c(0,0)$ of such an $f$ is given in terms of the coefficients of the Eisenstein series $E_{\kappa,L}\in M_\kappa(\rho_L)$ by
\begin{align}
\label{eq:ctcond}
c(0,0)=-\sum_{\nu\in L'/L}\sum_{l>0}  c(-l,\nu)\, e_{\kappa,L}(l,\nu) .
\end{align}
\end{proposition}

\begin{remark}
The condition on the Witt rank $r(L)$ in the second part of the proposition implies that the Eisenstein series $E_{\kappa,L}$ is holomorphic even in the case $n=2$.
\end{remark}

\begin{proof}[Proof of Proposition \ref{prop:keyprop}]
We generalize the argument of \cite[Lemma 4.11]{BBK}.
According to Theorem~\ref{thm:intbasis}, 
the space $M_{2-\kappa}^!(\rho_{L^-})$ has a basis of weakly holomorphic modular forms with integral coefficients. Hence, it suffices to show the existence of an $f\in M_{2-\kappa}^!(\rho_{L^-})$ with {\em rational} coefficients satisfying  (i) and (ii).
Let us first assume that $n>2$ or $n=2>r(L)$, so that \eqref{eq:ctcond} holds.

To lighten the notation, throughout the proof we write $M_\kappa$ for the $\Q$-vector space of holomorphic modular forms in $M_\kappa(\rho_L)$ with rational coefficients.
We write $S_\kappa$ for the subspace of cusp forms with rational coefficients.
The $\Q$-dual spaces are denoted by $M_\kappa^\vee$ and $S_\kappa^\vee$, respectively.
The natural inclusion $S_\kappa\to M_\kappa$ induces a surjective linear map
\[
\pr: M_\kappa^\vee\to S_\kappa^\vee ,\quad a\mapsto \pr (a).
\]
For $\mu\in L'/L$ and $m\in \Z+Q(\mu)$, we write
$a_{m,\mu}$ for the element in $M_\kappa^\vee$ taking an element
$g=\sum_\nu\sum_l b(l,\nu)q^l\chi_\nu\in M_\kappa$ to the Fourier coefficient
\[
a_{m,\mu}(g)= b(m,\mu).
\]
We let $M_{\kappa,\calS}^\vee\subset M_\kappa^\vee$ be the subspace generated by the functionals $a_{m,\mu}$ with $(m,\mu)\in \calS$.
According to Proposition \ref{prop:crit}, it suffices to show that there exists an
$a\in M_{\kappa,\calS}^\vee$ with $\pr(a)=0$ and $a(E_{\kappa,L})\neq 0$.

Let $a_1,\dots,a_d\in M_{\kappa,\calS}^\vee$ such that
$\pr(a_1),\dots,\pr(a_d)$ is a basis of $\pr(M_{\kappa,\calS}^\vee)\subset S_\kappa^\vee$.
Then for every $(m,\mu)\in \calS$ there exists a unique 
vector
$r(m,\mu)=(r_1(m,\mu),\dots,r_d(m,\mu))\in \Q^d$ such that
\[
\pr(a_{m,\mu}) = r_1(m,\mu)\cdot \pr(a_1)+\ldots+r_d(m,\mu)\cdot \pr(a_d).
\]
The linear combination
\begin{align}
\label{eq:cusp}
\tilde a_{m,\mu}:= a_{m,\mu} - r_1(m,\mu)\cdot a_1-\ldots-r_d(m,\mu)\cdot a_d \in M_{\kappa,\calS}^\vee
\end{align}
is in the kernel of $\pr$.
Evaluating $\tilde a_{m,\mu}$ at the Eisenstein series, we obtain
\begin{align}
\label{eq:eis}
\tilde a_{m,\mu}(E_{\kappa,L}) = e_{\kappa,L}(m,\mu)-r_1(m,\mu)\cdot a_1(E_{\kappa,L})-\ldots-r_d(m,\mu)\cdot a_d(E_{\kappa,L}).
\end{align}
In view of \eqref{eq:ctcond}, it suffices to show that there is an $(m,\mu)\in \calS$ such that $\tilde a_{m,\mu}(E_{\kappa,L})$ 
is non-zero.

To see this,
we assume on the contrary that $\tilde a_{m,\mu}(E_{\kappa,L})=0$ for all $(m,\mu)\in \calS$.
 We let $\|r\|$ be the euclidian norm of a vector $r\in \R^d$. Moreover, we also denote by $\|\cdot\|$ a norm on $S_\kappa^\vee\otimes \R$, say the operator norm.
Since $\pr(a_1),\dots,\pr(a_d)$ are linearly independent, there exists an $\eps>0$ such that
\[
\| r_1 \pr(a_1)+\ldots+r_d  \pr(a_d)\| \geq \eps \|r\|
\]
for all $r=(r_1,\dots,r_d)\in \R^d$.
By means of \eqref{eq:cusp} we obtain
\begin{align}
\label{eq:1}
\| \pr (a_{m,\mu}) \| \geq \eps \cdot \|r(m,\mu)\|.
\end{align}
On the other hand, our assumption $\tilde a_{m,\mu}(E_{\kappa,L})=0$ and \eqref{eq:eis} imply
that there is a constant $C'>0$ such that
\begin{align}
\label{eq:2}
|e_{\kappa,L}(m,\mu)|\leq C'\cdot \|r(m,\mu)\|.
\end{align}
Together, \eqref{eq:1} and \eqref{eq:2} imply that
\begin{align}
\label{eq:3}
|e_{\kappa,L}(m,\mu)| \leq\frac{C'}{\eps}\cdot \| \pr (a_{m,\mu}) \|
\end{align}
for all $(m,\mu)\in \calS$.
The Weil bound for the coefficients of (scalar valued) cusp forms of weight $\kappa$ for $\Gamma(N)$ implies that
$\| \pr (a_{m,\mu}) \| = O(m^{\kappa/2-1/4+\delta})$
as $m\to \infty$ for any $\delta>0$. Combining this with \eqref{eq:3} we obtain
\[
|e_{\kappa,L}(m,\mu)| = O(m^{\kappa/2-1/4+\delta})
\]
for $(m,\mu)\in \calS$ and $m\to \infty$, contradicting Proposition \ref{prop:eislb} and Remark \ref{rem:eislb}.

We finally consider the remaining case $n=2=r(L)$. From \eqref{eq:cuspcond} and the fact that $\calS$ is infinite, we easily deduce the existence of an $f$ satisfying condition (i) of Proposition \ref{prop:keyprop}, but possibly violating condition (ii). The fact that $r(L)=2$ implies that 
there is an even overlattice $ M \supset L$ which is isomorphic to the even unimodular lattice $I\!I_{2,2}$ of signature $(2,2)$.
%
%
This in turn implies that $\C[L'/L]$ contains a rational vector $f_0$ which is invariant under the Weil representation $\rho_{L^-}$ and which has non-zero $\chi_0$-component. In other words, $f_0$ is a non-zero element of $M_0(\rho_{L^-})$.
A suitable linear combination of $f$ and $f_0$ satisfies both conditions of 
  Proposition \ref{prop:keyprop}.
 \end{proof}

\subsection{Proof of Theorem \ref{prop:quot}}

Throughout this subsection we assume
that $(L,Q)$ has signature $(n,2)$ with $n\geq 1$.
We briefly write $L^-$ for the lattice $(L,-Q)$ of signature $(2,n)$.

\begin{lemma}
\label{lem:eisl}
Let $b\in \Z$ such that $k:=1-n/2+12b $ is greater than $2$.
The Eisenstein series $E_{k,L^-}\in M_k(\rho_{L^-})$ has non-negative Fourier coefficients  $e_{k,L^-}(l,\mu)$. When $l\in -
Q(\mu+L)$ is positive, the coefficient $e_{k,L^-}(l,\mu)$ is strictly positive.
\end{lemma}

\begin{proof}
The non-negativity of the coefficients is a direct consequence of Proposition \ref{prop:nonneg}. When $l\in -
Q(\mu+L)$, then the congruence representation numbers $N_{l,\mu}(p^\nu)$ for the lattice $L^-$ are all positive, since there even exists a global solution.
Therefore the claimed positivity follows from Theorem \ref{thm:eiseven} and Theorem \ref{thm:eisodd}.
\end{proof}

For $\mu\in L'/L$ we define 
\begin{align*}
t_\mu&= \min\{ -Q(\lambda)\mid \;\text{$\lambda\in \mu+L$ and $-Q(\lambda)>0$}\},\\
T&= \max\{ t_\mu\mid\; \mu\in L'/L\}.
\end{align*}
Since $L$ is indefinite, $t_\mu$ has a finite value in $\frac{1}{N}\Z_{>0}$.
The coefficient $e_{k,L^-}(t_\mu,\mu)$ of the Eisenstein series $E_{k,L^-}$ of Lemma \ref{lem:eisl} is positive (for any choice of $b$).
%

\begin{lemma}
\label{lem:pospp}
Let $b\in \Z_{>0}$ such that $k:= 1-n/2+12b$ is greater than $2$. There exists an element  $h\in M^!_{1-n/2}(\rho_{L^-})$ with non-negative rational Fourier coefficients $c_h(l,\mu)$ such that
\[
c_h(l,\mu) >0
\]
for all $\mu\in L'/L$ and all $l\in \Z-Q(\mu)$ with $l\geq T-b$.
\end{lemma}

\begin{proof}
Let $E_{k,L^-}\in M_k(\rho_{L^-})$ be the Eisenstein series of weight $k$ of Lemma \ref{lem:eisl}.
Then
\[
h(\tau)=\Delta(\tau)^{-b} E_{k,L^-}(\tau)
\]
belongs to $M_{1-n/2}^!(\rho_{L^-})$.
The product expansion $\Delta=q\prod_{j\geq 1}(1-q^j)^{24}$ of the discriminant function implies that the Fourier coefficients $c_{\Delta^{-1}}(j)$ of $\Delta^{-1}$ with index $j \geq -1$ are all positive.
Consequently, the coefficients of 
$c_{\Delta^{-b}}(j)$ of $\Delta^{-b}$ with index $j \geq -b$ are all positive. By Lemma \ref{lem:eisl} we obtain that the coefficients of $h$  are all non-negative.

If $\mu\in L'/L$ and $l\in \Z-Q(\mu)$ with $l\geq T-b$, we have that
\begin{align*}
c_h(l,\mu) &= \sum_{j\in \Z} c_{\Delta^{-b}}(j) \cdot e_{k,L^-}(l-j,\mu) \\
&= c_{\Delta^{-b}}(l-t_\mu) \cdot e_{k,L^-}(t_\mu,\mu) 
+\sum_{\substack{j\in \Z\\ j\neq l-t_\mu}} c_{\Delta^{-b}}(j) \cdot e_{k,L^-}(l-j,\lambda). 
\end{align*}
The hypothesis $l\geq T-b$ implies that $l-t_\lambda\geq -b$, and therefore,  by Lemma \ref{lem:eisl}, the first quantity on the right hand side of the latter equation is positive. Since the second quantity is non-negative, we obtain the assertion.
\end{proof}

We say that a weakly holomorphic modular form  $f\in M_k^!(\rho_{L^-})$ with Fourier coefficients $c_f(l,\mu)$ has {\em non-negative principal part\/} 
if
\[
c_f(l,\mu)\geq 0
\]
for all $\mu\in L'/L$ and all $l<0$.
Note that the Borcherds lift $\Psi(z,f)$ of any $f\in M_{1-n/2}^!(\rho_{L^-})$ with integral and non-negative principal part is holomorphic on $X_\Gamma$. Theorem~\ref{prop:quot} is a direct consequence of the following proposition.

\begin{proposition}
Let $f\in M_{1-n/2}^!(\rho_{L^-})$. There exist $f_1,f_2\in M_{1-n/2}^!(\rho_{L^-})$
with non-negative principal part such that
$f=f_1-f_2$. If $f$ has integral principal part, we may also choose $f_1$ and $f_2$ with integral principal part.
\end{proposition}

\begin{proof}
Let $b\in \Z_{>0}$ such that $b-T$ is greater than the order of the pole of $f$ at $\infty$. Let $h$ be the corresponding element of  $M_{1-n/2}^!(\rho_{L^-})$ as in Lemma \ref{lem:pospp}. Then there exists a positive integer $c$ such that
\[
f_1=f+c\cdot h
\]
has non-negative principal part. Setting in addition $f_2=c\cdot h$
gives the desired representation of $f$.
\end{proof}

We close this section with a variant of Theorem~\ref{prop:quot}, which can be proved similarly.

\begin{theorem}
\label{prop:quot2}
For every $\mu\in L'/L$ and every positive $m\in Q(\mu+L)$ there exists a (non-zero) holomorphic Borcherds product for $\Gamma$ which vanishes along $Z(m,\mu)$.
\end{theorem}

In view of \cite[Theorem 6.3]{HM} and \cite[Remark 7.2]{HM} this result can be employed to prove the converse theorem  for Borcherds products for lattices that split two hyperbolic planes over $\Z$ (see \cite[Theorem 5.12]{Br1}) in a completely different way.


\begin{thebibliography}{AGHM}



\bibitem[AGHM]{AGHM} {\em M. F. Andreatta, E. Goren, B. Howard, and K. Madapusi Pera}, Faltings heights of abelian varieties with complex multiplication, preprint (2015),  arXiv:1508.00178 .



\bibitem[Bo1]{Bo1}
\emph{R. E. Borcherds}, Automorphic forms with singularities on
Grassmannians, Invent. Math. \textbf{132} (1998), 491--562.

\bibitem[Bo2]{Bo2}
 \emph{R. Borcherds}, The Gross-Kohnen-Zagier theorem in higher
dimensions, Duke Math. J. \textbf{97} (1999), 219--233.
Correction in: Duke Math J. \textbf{105}, 183--184.


\bibitem[Br]{Br1} \emph{J. Bruinier}, Borcherds products on
  $\Orth(2,l)$ and Chern classes of Heegner divisors, Springer Lecture
  Notes in Mathematics {\bf 1780}, Springer-Verlag (2002).

\bibitem[BBK]{BBK}
\emph{J. Bruinier, J. Burgos, and U. K\"uhn},
Borcherds products and arithmetic intersection theory on Hilbert modular surfaces, Duke Math. Journal {\bf 139} (2007), 1--88.

\bibitem[BF]{BF}
\emph{J. Bruinier and J. Funke},
On two geometric theta lifts,
Duke Math J. \textbf{125} (2004), 45-90.

\bibitem[BHY]{BHY} \emph{J. Bruinier, B. Howard, and T. Yang},
Heights of Kudla-Rapoport divisors and derivatives of L-functions, Invent. Math. {\bf 201} (2015), 1--95.

\bibitem[BrK\"u]{BrKue}
\emph{J. Bruinier and U. K\"uhn},
Integrals of automorphic Green's functions associated to Heegner divisors (with U. Kühn), Int. Math. Res. Not. {\bf 2003:31} (2003), 1687--1729.

\bibitem[BK]{BK}
\emph{J. Bruinier and M. Kuss},
Eisenstein series attached to lattices and modular forms on orthogonal groups (with M. Kuss), Manuscr. Math. {\bf 106} (2001), 443--459.


\bibitem[HM]{HM} \emph{B. Heim and A. Murase}, A Characterization of Holomorphic Borcherds Lifts by Symmetries, Int. Math. Res. Not.  (2015) {\bf 2015:21} (2015), 11150--11185. 

\bibitem[Ku1]{Ku:Duke}  \emph{S. Kudla},
Algebraic cycles on Shimura varieties of orthogonal type.  Duke
Math. J.  {\bf 86}  (1997), 39--78.

\bibitem[Ku2]{Ku:Integrals}
\emph{S. Kudla}, Integrals of Borcherds forms, Compositio Math.
\textbf{137} (2003), 293--349.

\bibitem[KY]{KY} \emph{S. Kudla and T. Yang}, Eisenstein series for $SL(2)$, Sci. China Math. {\bf 53} (2010), 2275--2316.


\bibitem[McG]{McG} \emph{W. J. McGraw}, The rationality of vector valued modular forms associated with the Weil representation,  Math. Ann.  {\bf 326}  (2003), 105--122.








\end{thebibliography}
\end{document}